\documentclass{article}
\usepackage[utf8]{inputenc}
\usepackage{amssymb,amsmath,amsthm}
\usepackage{setspace,graphicx,color,hyperref}
\usepackage{enumerate}
\usepackage{authblk}

\newcommand{\Q}[0]{\mathbb{Q}}
\newcommand{\Z}[0]{\mathbb{Z}}
\newcommand{\R}[0]{\mathbb{R}}

\newcommand{\knap}[0]{R}

\newtheorem{lemma}{Lemma}
\newtheorem{theorem}{Theorem}
\newtheorem{corollary}{Corollary}

\newtheorem{definition}{Definition}

\DeclareMathOperator{\conv}{conv}

\DeclareMathOperator{\cone}{cone}

\DeclareMathOperator{\agg}{\mathcal{A}}

\title{The Aggregation Closure is Polyhedral for Packing and Covering Integer Programs}
\author[1]{Kanstantsin Pashkovich}
\author[2]{Laurent Poirrier}
\author[2]{Haripriya Pulyassary}
    
\affil[1]{University of Ottawa, School of Computer Science and Electrical Engineering, \protect\\
	\textit {kpashkov@uottawa.ca}}
\affil[2]{University of Waterloo, Department of Combinatorics and Optimization, \protect\\
	\textit{\{lpoirrier, hpulyassary\}@uwaterloo.ca}}

\date{\today}

\begin{document}

\maketitle

\abstract{
Recently, Bodur, Del Pia, Dey, Molinaro and Pokutta introduced the concept of \emph{aggregation cuts} for packing and covering integer programs. The aggregation \emph{closure} is the intersection of all aggregation cuts. Bodur et.\@ al.\@ studied the strength of this closure, but left open the question of whether the aggregation closure is polyhedral. In this paper, we answer this question in the positive, i.e. we show that the aggregation closure is polyhedral. Finally, we demonstrate that a generalization, the $k$-aggregation closure, is also polyhedral for all~$k$.
}

\parskip \baselineskip
\parindent 0cm

\section{Introduction}

A \emph{packing} integer programming (IP) problem optimizes a linear objective function over a set of the form
\[
P = \{ x \in \Z^n \, : \, x \geq 0 \text{ and } A x \leq b \},
\]
where $A \in \Z^{m \times n}$, $b \in \Z^m$, $A_{ij} \geq 0$ for
all $i= 1, \ldots, m$, $j=1, \ldots, n$ and
$b_i > 0$ for all $i=1, \dots, m$.
An \emph{aggregation} of $P$ is a knapsack set
\[
\knap^{\lambda} =
\{ x \in \Z^n \, : \, x \geq 0 \text{ and } \lambda^T A x \leq \lambda^T b \},
\]
where $\lambda^T A x \leq \lambda^T b$ is a nonnegative combination of the constraints $A x \leq b$, for some $\lambda \in \R^m$,
$\lambda \geq 0$. An \emph{aggregation cut} for $P$ is any inequality
that is valid for an aggregation of $P$. The aggregation \emph{closure} of $P$
is the intersection of all aggregation cuts for $P$. The aggregation closure
is defined in a similar way for \emph{covering} IPs, i.e. for problems over sets of the form
$P = \{ x \in \Z^n \, : \, x \geq 0 \text{ and } A x \geq b \}$.

The concept of aggregation was introduced
by Bodur, Del Pia, Dey, Molinaro and Pokutta in~\cite{Bodur2018}.
Like them, we formally define aggregations only for packing
and covering IPs.
Informally, though, the concept can be extended to general IPs,
in which case the resulting cuts are generally known as
knapsack cuts~\cite{Fukasawa2011}.
Knapsack cuts encompass a wide range of (overlapping) cutting plane families:
Chvatál-Gomory cuts~\cite{Gomory1958,Gomory1963,Chvatal1973},
Gomory Mixed-Integer cuts~\cite{Gomory1960},
lifted cover inequalities~\cite{Crowder1983,Wolsey1975,Zemel1978,Gu1998},
Mixed-Integer Rounding cuts~\cite{Nemhauser1988,Nemhauser1990,Marchand2001},
split cuts~\cite{Cook1990},
lift-and-project cuts~\cite{Balas1993,Balas1998,Balas2003},
group cuts~\cite{Gomory1972a,Gomory1972b,Johnson1974,Dash2008},
one-row cuts~\cite{Fukasawa2018}.
In addition, the concept of aggregation can be generalized to relaxations
$\knap^{\Lambda} =
\{ x \in \Z^n \, : \, x \geq 0 \text{ and } \Lambda^T A x \leq \Lambda^T b \}
$ where $\Lambda \in \R^{m \times k}$ and $\Lambda \geq 0$, i.e.,
multi-row relaxations of $P$. Bodur et.\@ al.\@~\cite{Bodur2018}
call such sets $k$-aggregations, and their general IP counterpart are the
so-called \emph{multi-row cuts}, which have been a prolific area of reseach
over the last decade~\cite{Conforti2014}.

We focus on packing and covering IPs,
because keeping the fixed set of constraints $x \geq 0$
in all relaxations $\knap^{\lambda}$ of $P$
is the defining feature of aggregation cuts as defined above.
Indeed, consider the sets
$\knap'^{\lambda} = \{ x \in \Z^n \, : \, \lambda^T A x \leq \lambda^T b \}$
obtained by dropping the nonnegativity requirements from $\knap^{\lambda}$.
For $\lambda \in \Q^m$, $\lambda \geq 0$, the integer hulls $\conv(\knap'^{\lambda})$
are exactly the
Chvatál-Gomory cuts of $P$.
It is thus clear that aggregation cuts dominate Chvatál-Gomory cuts,
implying that they should yield strong valid inequalities.
Interestingly, however,
Bodur et.\@ al.\@~\cite{Bodur2018} show that the aggregation closure
can be 2-approximated by intersecting the integer hulls of knapsacks
$\knap^{e_i}$ for $i=1, \ldots, m$.
Here, $e_j$ denotes the $j$-th vector in the standard basis of~$\R^m$.

Given any infinitely-generated family of cutting planes,
it is natural to ask whether its \emph{closure}, for
a polyhedral formulation, is a polyhedron again.
Examples of such polyhedral closures include
the Chvátal-Gomory closure~\cite{Schrijver1980}
and the split closure~\cite{Cook1990}.
Bodur et.\@ al.\@~\cite{Bodur2018} show that if $A$ is fully dense, i.e.,
if $A_{ij} > 0$ for all $i =1, \ldots, m$, $j= 1, \ldots, n$,
then the aggregation closure is polyhedral.
They state as an open problem the
question of whether the aggregation closure is polyhedral in the general case.

Here, we answer the latter question in the positive.
Note that we were recently made aware that an independent proof
was the subject of a poster by Del Pia, Linderoth and Zhu,
presented at the MIP 2019 workshop~\cite{DelPia2019}.

\section{Our Technique}
\label{sect:technique}

Bodur et.\@ al.\@~\cite{Bodur2018} showed that the aggregation closure is polyhedral whenever the inequalities defining the integer programming problem are dense. This result was obtained by showing that, in the dense case, all vertices of any aggregation belong to some finite set of points. Thus, there exists only finitely many possible aggregation cuts, and so the aggregation closure is an intersection of finitely many halfspaces.

We prove that the aggregation closure is a polyhedron by induction on the dimension. First, we consider inequalities that are not dense. The fact that these inequalities are not dense allows us to make use  of the inductive hypothesis. In this way, we are able to obtain a finite set of sparse ineaqualities that are valid for the aggregation closure. Second, we obtain a finite set of dense inequalities. Finally, we show that the finite set of the obtained sparse and dense inequalities is enough to define the aggregation closure, showing that the aggregation closure is polyhedral.

\section{Packing Integer Programs}
\label{sect:packing}

The first object of our study are feasible regions of packing integer programs,
i.e. sets of the form
\[
\{ x \in \Z^n \, : \, x \geq 0 \text{ and } A x \leq b \}\,,
\]
where $A\in \Z^{m\times n}$, $b\in \Z^m$ and $A_{ij} \geq 0$ and $b_i>0$ for all $i= 1, \ldots, m$, $j=1, \ldots, n$.

Let us consider the corresponding relaxation
\[
Q := \{ x \in \R^n \, : \, x \geq 0 \text{ and } A x \leq b \}\,.
\]
Observe that the zero vector lies in $Q$ so $Q$ is not empty.
Furthermore, for simplicity of exposition,
we assume that $Q$ is not trivial, i.e. $Q \neq \R^n_+$.

For every $\lambda \in \R^m_+$, we can define the following relaxation of $Q$ given by one aggregated inequality corresponding to $\lambda$
\[
Q^{\lambda} :=
\{ x \in \R^n \, : \, x \geq 0 \text{ and } \lambda^T A x \leq \lambda^T b \}.
\]
We define the aggregation closure $\agg(Q)$ of the above formulation of $Q$ as follows
\[
\agg(Q) := \bigcap_{\lambda \in \R^m_+}  \conv\{ Q^{\lambda} \cap \Z^n \}.
\]
Furthermore, we can write
$\agg(Q^\lambda): = \conv\{ Q^{\lambda} \cap \Z^n \}$, and so $\agg(Q)=\bigcap_{\lambda \in \R^m_+} \agg(Q^\lambda)$.

Our first main result is the following theorem.
\begin{theorem}\label{thm:main_packing}
For every packing integer program, the aggregation closure is a polyhedron.
\end{theorem}

To show that $\agg(Q)$ is polyhedral, we proceed by induction on the dimension $n$. The case $n = 1$ is straightforward. 

 Let us assume that $\agg(Q)$ is polyhedral for any packing integer program whenever $Q\subseteq \R^n$ and $n=1,\ldots, \ell$ for an integer $\ell$, $\ell\geq 1$. We will show that $\agg(Q)$ is polyhedral also for $n=\ell+1$.

For $j=1,\ldots,n$, define  
\[
Q_j \; := \; \{ x \in \R^n \, : \,  x \geq 0,\, x_j=0 \text{ and } A x \leq b\}
    \; + \; \cone(e_j).
\]
\begin{lemma}
The set
\[
L := \bigcap_{j = 1, \ldots, n} \agg(Q_j)
\]
is polyhedral.
\end{lemma}
\begin{proof}
For every $j=1,\ldots, n$, let us define
\[
Q'_j:=\{ x \in \R^{n-1} \, : \,  x \geq 0 \text{ and } A'x \leq b\}\,,
\]
where $A'\in \R^{m\times(n-1)}$ is obtained from $A\in \R^{m\times n}$ by dropping $j$-th column. Note that
\[
Q_j=\R_+\times Q'_j\,,
\]
and hence
\[
\agg(Q_j)=\R_+\times \agg(Q'_j).
\]
By the inductive hypothesis $\agg(Q'_j)$ is polyhedral for every $j=1,\ldots, n$.
Therefore, $\agg(Q_j)$ is a polyhedron for every $j=1,\ldots, n$, showing that $L$ is an intersection of a finite number of polyhedra. Thus $L$ is also polyhedral.
\end{proof}

Observe that for any packing polyhedra $G', G'' \subseteq \R^n_+$, $G' \subseteq G''$
implies $\agg(G') \subseteq \agg(G'')$. In particular, we have
$Q_j \supseteq Q$ for all $j = 1, \ldots, n$, hence
$L \supseteq \agg(Q)$.
\begin{lemma}\label{lem:point_minus_orthant}
For every point $\tilde{x} \in L$ and for every $j = 1, \ldots, n$ there exists $\gamma\in \R_+$ such that for every $\lambda \in \R^m_+$, we have
\[
\tilde{x} - \gamma e_j \in \agg( Q^{\lambda} )\,.
\]
\end{lemma}
\begin{proof}
The proof follows from definition of $L$. In particular, given a point $\tilde{x}\in L$ and $j = 1, \ldots, n$, we can take $\gamma=\tilde{x}_j$.
Since $\tilde{x} \in \agg(Q_j)$, we have $\tilde{x} - \gamma e_j \in \agg(Q'_j)\times\{0\}$.
At which point $Q'_j \subseteq Q$ yields $\agg(Q'_j)\times\{0\} \subseteq \agg(Q)$
thus $\tilde{x} - \gamma e_j \in \agg(Q)$.
\end{proof}
We proceed by introducing a standard concept of \emph{domination} for points in $\R^n$.
\begin{definition}
We say that $x\in \R^n$ \emph{dominates} $y\in \R^n$, denoted by $x\succeq y$, if $x_j\geq y_j$ for every $j=1,\ldots,n$. Moreover, if $x\succeq y$ but $x\neq y$, we say that $x$ strictly dominates $y$ and denote it by $x\succ y$.
\end{definition}

\begin{lemma}\label{lem:dominant_set_finite}
Let $S$ be a set of points in $\Z^n_+$, such that there exists no two points $x,y\in S$ with $x\succ y$. Then $S$ is finite.
\end{lemma}
\begin{proof}
Let us prove the statement by induction on $n$. The base case $n=1$ is straightforward, because there exists a minimum number in every non empty set of integer non-negative numbers. Hence, if $n=1$ the set~$S$ is either empty or consists of a single point.

Let us assume that the statement is true for all $n=1,\ldots, \ell$, $\ell\geq1$. We will show that the statement holds also for $n=\ell+1$. If $S$ is empty then $S$ is finite. If $S$ contains a point $x'\in \Z^n_+$, then $S$ can be represented as follows
\begin{align*}
S=\{x'\}\bigcup&\big(\cup_{t=0,\ldots, x'_1-1} \{x\in S\,:\, x_1=t \}\big)\bigcup  \big(\cup_{t=0,\ldots, x'_2-1} \{x\in S\,:\, x_2=t \}\big)\bigcup\ldots\\
	\bigcup&\big(\cup_{t=0,\ldots, x'_n-1} \{x\in S\,:\, x_n=t \}\big)\,,
\end{align*}
because no point in $S\setminus\{x'\}$ can  dominate $x'$.

 Note that for each $j=1,\ldots, n$ and $t\in \Z_+$ the set $\{x\in S\,:\, x_j=t \}$ is finite by the inductive hypothesis. Hence, $S$ is a union of a finite number of finite sets, and is thus a finite set.
\end{proof}

\begin{definition}
We say that an $n$-tuple of points $(x',x'',\ldots, x^{(n)})$, $x',x'',\ldots, x^{(n)}\in\R^n$ dominates an $n$-tuple $(y',y'',\ldots, y^{(n)})$, $y',y'',\ldots, y^{(n)}\in\R^n$, denoted by $(x',x'',\ldots, x^{(n)})\succeq (y',y'',\ldots, y^{(n)})$, if $x'\succeq y'$, $x''\succeq y''$,\ldots, $x^{(n)}\succeq y^{(n)}$. Furthermore, if $(x',x'',\ldots, x^{(n)})\succeq (y',y'',\ldots, y^{(n)})$ and $(x',x'',\ldots, x^{(n)})\neq (y',y'',\ldots, y^{(n)})$, we say that $(x',x'',\ldots, x^{(n)})$ strictly dominates $(y',y'',\ldots, y^{(n)})$ and we denote it by  $(x',x'',\ldots, x^{(n)})\succ (y',y'',\ldots, y^{(n)})$.
\end{definition}

The following corollary follows from Lemma~\ref{lem:dominant_set_finite}.

\begin{corollary}\label{cor:dominant_set_finite_tuples}
Let $S$ be a set of $n$-tuples of points in $\Z^n_+$, such that there exists no two $n$-tuples $(x',x'',\ldots, x^{(n)})\in S$ and $(y',y'',\ldots, y^{(n)})\in S$ with $(x',x'',\ldots, x^{(n)})\succ (y',y'',\ldots, y^{(n)})$. Then $S$ is finite.
\end{corollary}

Let us now construct, aside from the polyhedron $L$, another polyhedron $K$. Later, we will prove that $\agg(Q)=L\cap K$ and thus show that $\agg(Q)$ is a polyhedron.

Consider the following set $T$ of $n$-tuples of points in $\Z_+^n$
\begin{align*}
T:=\{&(x',x'',\ldots, x^{(n)})\in \Z^n_+\times\Z^n_+\times\ldots\times\Z^n_+ \,:\, \\
&\text{there is  } \lambda \in \R^m_+ \text{ and } \\
&\text{there is a facet-defining inequality } a^Tx\leq \beta \text{ of  } \agg(Q^\lambda) \text{ with } a_j > 0 \text{ for all } j = 1, \ldots, n \\
&\text{such that } x',x'',\ldots, x^{(n)}\text{ are affinely independent and }\\
&a^T x'= \beta, a^T x''= \beta,\ldots,a^T x^{(n)}= \beta\}\,,
\end{align*}
and define 
\begin{align*}
S:= \{(x',x'',\ldots, x^{(n)})\in T\,:\,
&\text{ there is no }(y',y'',\ldots, y^{(n)})\in T\\
&\text{ such that }(y',y'',\ldots, y^{(n)})\prec (x',x'',\ldots, x^{(n)})\}\,.
\end{align*}
By Corollary~\ref{cor:dominant_set_finite_tuples}, $S$ is a finite set. 

Now, for each $n$-tuple $(x',x'',\ldots, x^{(n)})$ in $S$ we take the corresponding inequality to define $K$ as follows
\begin{align*}
K:= \{ x\in \R^n\,:\, a^T x\leq \beta, & \text{ for all } a\in \R^n \text{ and } \beta=1 \\ 
&\text{ such that } a^T x'= \beta, a^T x''= \beta,\ldots,a^T x^{(n)}= \beta\\
&\text{ for some }(x',x'',\ldots, x^{(n)})\in S\}\,.
\end{align*}
Since $S$ is finite and all $n$-tuples in $S$ are formed by affinely independent points, $K$ is a polyhedron.

\begin{lemma}
We have $K\cap L=\agg(Q)$.
\end{lemma}
\begin{proof}
As mentioned earlier, we have $\agg(Q)\subseteq L$.
Furthermore, since $K$ is formed by considering inequalities valid for $\agg(Q^\lambda)$ for some $\lambda\in \R^m_+$, we have $\agg(Q)\subseteq K$. Therefore, $\agg(Q)\subseteq K\cap L$, so it is enough to prove  $K\cap L\subseteq \agg(Q)$. Let us assume that there exists $x^\star$ such that $x^\star\not\in \agg(Q)$ and $x^\star \in L$, and let us show that in this case $x^\star \not\in K$.

Since $x^\star$ does not belong to $\agg(Q)$, we have that $x^\star$ does not belong to $\agg(Q^\lambda)$ for some $\lambda\in \R^m_+$.
By Lemma~\ref{lem:point_minus_orthant} and the convexity of $\agg(Q^\lambda)$, for each $r \in \R^n_+ \setminus \{ 0 \}$, we have that the ray $x^\star - \cone(r)$ intersects $\agg(Q^\lambda)$.
In particular, if we choose $r^\star\in \R^n_+$, where $r^\star_j>0$ for all $j=1,\ldots,n$, then there exists $t^\star\in \R$, $t^\star>0$ such that $x^\star-t^\star r^\star$ lies in a facet of $\agg(Q^\lambda)$, where the corresponding facet-defining inequality is violated by $x^\star$.
Consider the corresponding facet-defining inequality $a^T x\leq \beta$ of $\agg(Q^\lambda)$. This inequality has no coefficient equal to zero. Indeed, if for some $j \in [n]$ we have $a_j=0$ then we get $a^T(x^\star -\gamma e_j)=a^T x^\star>\beta$ and so $x^\star-\gamma e_j\not \in \agg(Q^\lambda)$ for every $\gamma$, contradicting Lemma~\ref{lem:point_minus_orthant}. 

Since $\agg(Q^\lambda)$ is integral, there exists an $n$-tuple $(x',x'',\ldots, x^{(n)})$ where $x', x'',\ldots, x^{(n)}\in \Z^n_+$ are affinely independent, $a^T x'= \beta$, $a^T x''= \beta$, \ldots, $a^T x^{(n)}= \beta$ and $x^\star-t^\star r^\star$ is a convex combination of  the points $x'$, $x''$, \ldots, $x^{(n)}$. Hence, $(x',x'',\ldots, x^{(n)})$ belongs to the set $T$ and $x^\star-t^\star r^\star=\mu_1 x'+\mu_2 x''+\ldots \mu_n x^{(n)}$ for some $\mu_1$, $\mu_2$, \ldots, $\mu_n\in \R_+$ such that $\mu_1+\mu_2+\ldots \mu_n=1$.

Thus, there exists an $n$-tuple $(y',y'',\ldots, y^{(n)})$ in $S$ such that $(y',y'',\ldots, y^{(n)})\preceq (x',x'',\ldots, x^{(n)})$. Then, consider the inequality $a'^T x\leq \beta'$, where $\beta'=1$, defined by the tuple $(y',y'',\ldots, y^{(n)})$. The point $\mu_1 y'+\mu_2 y''+\ldots \mu_n y^{(n)}$ satisfies the inequality $a'^T x\geq \beta'$ at equality. Furthermore, 
\[
\mu_1 y'+\mu_2 y''+\ldots \mu_n y^{(n)}\preceq \mu_1 x'+\mu_2 x''+\ldots \mu_n x^{(n)}=x^\star-t^\star r^\star\,
\]
showing that $x^\star$ violates $a'^T x\leq \beta'$, since $a'_j>0$, $r^\star_j>0$ for all $j=1,\ldots,n$ and $t^\star>0$. Therefore, $x^\star$ does not belong to $K$.
\end{proof}

\bigskip
Given $k\in \Z$, $k>0$, for every $\Lambda \in \R^{m\times k}_+$ we can define the following relaxation of $Q$ given by $k$ aggregated inequalities corresponding to $\Lambda$
\[
Q^{\Lambda} :=
\{ x \in \R^n \, : \, x \geq 0 \text{ and } \Lambda^T A x \leq \Lambda^T b \}.
\]
We define the $k$-aggregation closure $\agg_k(Q)$ of the above formulation of $Q$ as follows
\[
\agg_k(Q) := \bigcap_{\Lambda \in \R^{m\times k}_+}  \conv\{ Q^{\Lambda} \cap \Z^n \}.
\]
The proof of Theorem~\ref{thm:main_packing_k} below  is fully analogous to the proof of Theorem~\ref{thm:main_packing}.
\begin{theorem}\label{thm:main_packing_k}
For every packing integer program, the $k$-aggregation closure is a polyhedron.
\end{theorem}

\section{Covering Integer Programs}
\label{sect:covering}

Let us now consider the feasible regions of covering integer programs, i.e. sets of the form
\[
\{ x \in \Z^n \, : \, x \geq 0 \text{ and } A x \geq b \}\,,
\]
where $A\in \Z^{m\times n}$, $b\in \Z^m$ and $A_{ij} \geq 0$ and $b_i>0$ for all $i= 1, \ldots, m$, $j=1, \ldots, n$.

We assume that the corresponding relaxation
\[
Q := \{ x \in \R^n \, : \, x \geq 0 \text{ and } A x \geq b \}\,.
\]
is not trivial, i.e. $Q$ is neither the empty set nor $\R^n_+$.

For every $\lambda \in \R^m_+$, we define an aggregation 
\[
Q^{\lambda} :=
\{ x \in \R^n \, : \, x \geq 0 \text{ and } \lambda^T A x \geq \lambda^T b \}.
\]
We define the aggregation closure $\agg(Q)$ of the above formulation of $Q$ as follows
\[
\agg(Q) := \bigcap_{\lambda \in \R^m_+}  \conv\{ Q^{\lambda} \cap \Z^n \}.
\]
Again, we can write
$\agg(Q^\lambda) := \conv\{ Q^{\lambda} \cap \Z^n \}$, and so $\agg(Q)=\bigcap_{\lambda \in \R^m_+} \agg(Q^\lambda)$.

Our second main result follows.
\begin{theorem}\label{thm:main_covering}
For every covering integer program, the aggregation closure is a polyhedron.
\end{theorem}

Again, we proceed by induction on the dimension $n$. The base case $n = 1$ is immediate.
Let us assume that $\agg(Q)$ is polyhedral for any covering integer program whenever $Q\subseteq \R^n$ and $n=1,\ldots, \ell$ for an integer $\ell$, $\ell\geq 1$. We will show that $\agg(Q)$ is polyhedral also for $n=\ell+1$.

For $j=1,\ldots,n$, define  $\mathcal{I}_j := \{i \in \{1,2,\ldots,m\} : A_i e_j = 0 \}$  and
\[
Q_j := \{ x \in \R^n \, : \,  x \geq 0 \text{ and } A_i x \geq b_i \text{ for all } i \in \mathcal{I}_j \}.
\]
Here, $A_i$ denotes the $i$-th row of $A$ for $i=1,\ldots,m$.
In other words, the index set $\mathcal{I}_j$ represents the rows of $A$
that have a zero in column $j$, and $Q_j$ is a relaxation of $Q$ corresponding to those rows.

\begin{lemma}
The set
\[
L := \bigcap_{j = 1, \ldots, n} \agg(Q_j)
\]
is polyhedral.
\end{lemma}
\begin{proof}
For every $j=1,\ldots, n$, let us define
\[
Q'_j:=\{ x \in \R^{n-1} \, : \,  x \geq 0 \text{ and } A'_i x \geq b_i \text{ for all } i \in \mathcal{I}_j\}\,,
\]
where $A'_i\in \R^{n-1}$ for all $i\in\mathcal{I}_j$ is obtained from $A_i\in \R^n$ by dropping the $j$-th entry. Note that
\[
Q_j=\R_+\times Q'_j\,,
\]
and hence
\[
\agg(Q_j)=\R_+\times \agg(Q'_j).
\]
By the inductive hypothesis $\agg(Q'_j)$ is polyhedral for every $j=1,\ldots, n$.
Therefore, $\agg(Q_j)$ is polyhedral for every $j=1,\ldots, n$, showing that $L$ is also polyhedral.
\end{proof}

Observe that for any covering polyhedra $G', G'' \subseteq \R^n_+$, $G' \subseteq G''$
implies $\agg(G') \subseteq \agg(G'')$. In particular,
$Q_j$ is a relaxation of $Q$ for all $j = 1, \ldots, n$, hence
$L \supseteq \agg(Q)$.
If $Q_j = Q$ for some $j$, then since $\agg(Q_j)$ is polyhedral as shown above,
$\agg(Q)$ is polyhedral as well and we are done. Thus, we now assume
that $\mathcal{I}_j \neq \{1,2,\ldots,m\}$ for all $j=1,\ldots,n$.

\begin{lemma}\label{lem:point_plus_orthant}
There exists $\gamma\in \Z$, $\gamma > 0$ such that, for every $\tilde{x} \in L$, for every $\lambda \in \R^m_+$,  for every $j = 1, \ldots, n$,
we have
\[
\tilde{x} + \gamma e_j \in \conv\{ Q^{\lambda} \cap \Z^n \}\,.
\]
\end{lemma}
\begin{proof}

Define 
\[
\gamma:=\left\lceil \max_{j=1,\ldots, n} \max_{i\not\in \mathcal{I}_j} \frac{b_i}{A_i^T e_j} \right\rceil\,.
\]
Note that $\gamma$ is well-defined and is a positive number.
Our choice of $\gamma$ implies
that for any $z \in \Z^n$, if $z_j \geq \gamma$ for some $j$,
then $A_t^T z \geq b_t$ for all $t \notin \mathcal{I}_j$, i.e.,
every constraint with a nonzero coefficient in column $j$ is satisfied.

Now consider $\tilde{x} \in L$, $\lambda \in \R^m_+$ and $j = 1, \ldots, n$. Since $\tilde{x}$ lies in $L$, we have $\tilde{x}\in \agg(Q_j)$, so
\[
\tilde{x}\in \conv\{x\in \Z^n\,:\, x\geq 0\text{ and }\sum_{i\in \mathcal{I}_j} \lambda_i A_i^Tx \geq \sum_{i\in \mathcal{I}_j} \lambda_i b_i\}\,.
\]
By the definition of $\gamma$, we have
\begin{align*}
\tilde{x}+\gamma e_j \in &\conv\{x+\gamma e_j\in \Z^n\,:\, x\geq 0\text{ and }\sum_{i\in \mathcal{I}_j} \lambda_i A_i^Tx \geq  \sum_{i\in \mathcal{I}_j}\lambda_i b_i\}\subseteq
\\
&\conv\{z\in \Z^n\,:\, z\geq 0\text{ and }\sum_{i\in \mathcal{I}_j} \lambda_i A_i^T z \geq \sum_{i\in \mathcal{I}_j} \lambda_i b_i \text{ and } A_t^T z\geq b_t \text{ for all } t\not\in \mathcal{I}_j\}\subseteq
\\
&\conv\{z\in \Z^n\,:\, z\geq 0\text{ and }\sum_{i=1,\ldots, m} \lambda_i A_i^T z \geq \sum_{i=1,\ldots, m} \lambda_i b_i\}=\\
&\conv\{ Q^{\lambda} \cap \Z^n \}\,,
\end{align*}
finishing the proof.
\end{proof}

\begin{corollary}\label{cor:point_plus_orthant_general}
There exist $\gamma\in \Z$, $\gamma > 0$ such that, for every $\tilde{x} \in L$ and  for every $j = 1, \ldots, n$,
we have
\[
\tilde{x} + \gamma e_j \in \agg(Q)\,.
\]
\end{corollary}

Consider the following set $T$ of $n$-tuples of points in $\Z_+^n$
\begin{align*}
T:=\{&(x',x'',\ldots, x^{(n)})\in \Z^n_+\times\Z^n_+\times\ldots\times\Z^n_+ \,:\, \\
&\text{there is  } \lambda \in \R^m_+ \text{ and } \\
&\text{there is a facet-defining inequality } a^Tx\geq \beta \text{ of  }\agg(Q^\lambda)
	\text{ with } a_j > 0 \text{ for all } j = 1, \ldots, n \\
&\text{such that } x',x'',\ldots, x^{(n)}\text{ are affinely independent and }\\
&a^T x'= \beta, a^T x''= \beta,\ldots,a^T x^{(n)}= \beta\}\,,
\end{align*}
and define 
\begin{align*}
S:=\{(x',x'',\ldots, x^{(n)})\in T\,:\,
&\text{ there is no }(y',y'',\ldots, y^{(n)})\in T\\
&\text{ such that }(y',y'',\ldots, y^{(n)})\succ (x',x'',\ldots, x^{(n)})\}\,.
\end{align*}
By Corollary~\ref{cor:dominant_set_finite_tuples}, $S$ is a finite set. Note that the set $S$ can be empty.

Now, for each $n$-tuple $(x',x'',\ldots, x^{(n)})$ in $S$ we take the corresponding inequality to define $K$ as follows
\begin{align*}
K:=\{ x\in \R^n\,:\, a^T x\geq \beta & \text{ for all } a\in \R^n \text{ and } \beta=1\,\\ 
&\text{ such that } a^T x'= \beta, a^T x''= \beta,\ldots,a^T x^{(n)}= \beta\\
&\text{ for some }(x',x'',\ldots, x^{(n)})\in S\}\,.
\end{align*}
Since $S$ is finite and all $n$-tuples in $S$ are formed by affinely independent points, $K$ is a polyhedron.

\begin{lemma}
We have $K\cap L=\agg(Q)$.
\end{lemma}
\begin{proof}
As mentioned earlier, we have $\agg(Q)\subseteq L$.
Furthermore, since $K$ is formed by considering a subset of
inequalities valid for $\agg(Q^\lambda)$ for some $\lambda\in \R^m_+$, we have $\agg(Q)\subseteq K$.
Therefore, $\agg(Q)\subseteq K\cap L$, so it is enough to prove  $K\cap L\subseteq \agg(Q)$.
For the sake of contradiction, let us assume that there exists $x^\star$ such that $x^\star \in K\cap L$ but $x^\star\not\in \agg(Q)$.

Since $x^\star$ does not belong to $\agg(Q)$, we have that $x^\star$ does not belong to $\agg(Q^\lambda)$ for some $\lambda\in \R^m_+$.
By Corollary~\ref{cor:point_plus_orthant_general},  for each $r \in \R^n_+ \setminus \{ 0 \}$, we have that the ray $x^\star+ \cone(r)$ intersects $\agg(Q^\lambda)$.
Hence, there exists $r^\star\in \R^n_+$, where $r^\star_j>0$ for all $j=1,\ldots,n$, and $t^\star\in \R$, $t^\star>0$ such that $x^\star+t^\star r^\star$ lies in the relative interior of a facet of $\agg(Q^\lambda)$ violated by $x^\star$.
Let this facet-defining inequality be $a^T x\geq \beta$. This inequality has no coefficient equal to zero. Indeed, if for some $j=1,\ldots,n$ we have $a_j=0$ then we get $a^T(x^\star+\gamma e_j)=a^T x^\star<\beta$ and so $x^\star+\gamma e_j\not \in \agg(Q^\lambda)$ for every $\gamma$, contradicting  Corollary~\ref{cor:point_plus_orthant_general}. 

Since $\agg(Q^\lambda)$ is an integral polyhedron, there exists an $n$-tuple $(x',x'',\ldots, x^{(n)})$ where $x', x'',\ldots, x^{(n)}\in \Z^n_+$ are affinely independent, $a^T x'= \beta$, $a^T x''= \beta$, \ldots, $a^T x^{(n)}= \beta$ and $x^\star+t^\star r^\star$ is a strict convex combination of  the points $x'$, $x''$, \ldots, $x^{(n)}$. Hence, $(x',x'',\ldots, x^{(n)})$ belongs to the set $T$ and $x^\star+t^\star r^\star=\mu_1 x'+\mu_2 x''+\ldots \mu_n x^{(n)}$ for some $\mu_1$, $\mu_2$, \ldots, $\mu_n\in \R$ such that $\mu_1+\mu_2+\ldots \mu_n=1$ and $\mu_1, \mu_2, \ldots, \mu_n>0$.

Now, we have two cases: there exists an $n$-tuple $(y',y'',\ldots, y^{(n)})$ in $S$ such that $(y',y'',\ldots, y^{(n)})\succeq (x',x'',\ldots, x^{(n)})$  or there exists no such an $n$-tuple in $S$. 

If there exists an $n$-tuple $(y',y'',\ldots, y^{(n)})$ in $S$ such that $(y',y'',\ldots, y^{(n)})\succeq (x',x'',\ldots, x^{(n)})$, then consider the inequality $a'^T x\geq \beta'$, where $\beta'=1$, defined by the tuple $(y',y'',\ldots, y^{(n)})$. The point $\mu_1 y'+\mu_2 y''+\ldots \mu_n y^{(n)}$ satisfies the inequality $a'^T x\geq \beta'$ at equality, and moreover 
\[
\mu_1 y'+\mu_2 y''+\ldots \mu_n y^{(n)}\succeq \mu_1 x'+\mu_2 x''+\ldots \mu_n x^{(n)}=x^\star+t^\star r^\star\,
\]
showing that $x^\star$ violates $a'^T x\geq \beta'$ because $a'_j>0$, $r^\star_j>0$ for all $j=1,\ldots,n$ and $t^\star>0$.  Thus $x^\star$ does not belong to $K$, a contradiction.

If there exists no $n$-tuple $(y',y'',\ldots, y^{(n)})$ in $S$ such that $(y',y'',\ldots, y^{(n)})\succeq (x',x'',\ldots, x^{(n)})$, then there exists an infinite sequence of $n$-tuples in $T$ such that the first $n$-tuple in the sequence dominates $(x',x'',\ldots, x^{(n)})$, and every $n$-tuple in the sequence is strictly dominated by the next $n$-tuple in the sequence. Hence, there exist $j=1,\ldots, n$, $k=1,\ldots,n$ and an infinite subsequence of this sequence  of $n$-tuples in $T$  with an additional condition that in every $n$-tuple the $j$-th coordinate of the $k$-th point of an $n$-tuple  in the subsequence is strictly smaller then the $j$-th coordinate of the $k$-th point of the next $n$-tuple in the subsequence.  Consider the sequence of points in $\R^n_+$ formed by taking the point 
\[
\mu_1 z'+\mu_2 z''+\ldots \mu_n z^{(n)}\,,
\]
using the above subsequence of $n$-tuples $(z',z'',\ldots, z^{(n)})$. Since $\mu_k>0$, we have that for every $\gamma\in \R$ there exists a point $\tilde{x}$ in this sequence of points such that $\tilde{x}\succ x^\star+\gamma e_j$. By the definition of $T$, this point $\tilde{x}$ lies on a facet of $\agg(Q^\zeta)$ for some $\zeta\in \R^m_+$, where the corresponding facet-defining inequality $\tilde{a}^Tx\geq \tilde{\beta}=1$ has no  coefficient equal to zero.
This shows that $x^\star+\gamma e_j$  does not belong to $\agg(Q^\zeta)$ because $\tilde{x}\succ x^\star+\gamma e_j$, which by Corollary~\ref{cor:point_plus_orthant_general} yields $x^\star \notin L$, a contradiction.
\end{proof}

\bigskip
Given $k\in \Z$, $k>0$, for every $\Lambda \in \R^{m\times k}_+$ we can define the following relaxation of $Q$ given by $k$ aggregated inequalities corresponding to $\Lambda$
\[
Q^{\Lambda} :=
\{ x \in \R^n \, : \, x \geq 0 \text{ and } \Lambda^T A x \geq \Lambda^T b \}.
\]
Again, we define the $k$-aggregation closure $\agg_k(Q)$ of the above formulation of $Q$ as follows
\[
\agg_k(Q) := \bigcap_{\Lambda \in \R^{m\times k}_+}  \conv\{ Q^{\Lambda} \cap \Z^n \}.
\]
The proof of Theorem~\ref{thm:main_covering_k} below  is fully analogous to the proof of Theorem~\ref{thm:main_covering}.
\begin{theorem}\label{thm:main_covering_k}
For every covering integer program, the $k$-aggregation closure is a polyhedron.
\end{theorem}

\section*{Acknowledgements}
We would like to thank Ricardo Fukasawa for informing us that an independent proof of polyhedrality for aggregation closures, by Alberto  Del  Pia,  Jeff  Linderoth,  and  Haoran  Zhu, was the subject of a poster at the Mixed Integer Programming Workshop, 2019. We also would like to thank Haoran Zhu for pointing us to the question about the polyhedrality of $k$-aggregation closures. 

\bibliographystyle{plain}
\bibliography{polyhedra}

\end{document}